\documentclass[a4paper]{amsart}
\usepackage{amsmath}
\usepackage{amsthm}
\usepackage{amsfonts}

\newcommand{\F}{\mathbb{F}}
\newcommand{\Q}{\mathbb{Q}}
\newcommand{\Z}{\mathbb{Z}}
\renewcommand{\P}{\mathbb{P}}
\newcommand{\A}{\mathbb{A}}
\newcommand{\X}{\mathcal{X}}
\newcommand{\Y}{\mathcal{Y}}
\newcommand{\kbar}{\bar{k}}
\renewcommand{\H}{\mathrm{H}}
\newcommand{\ad}{\mathbf{A}}

\DeclareMathOperator{\Br}{Br}
\DeclareMathOperator{\Pic}{Pic}
\DeclareMathOperator{\Gal}{Gal}
\DeclareMathOperator{\im}{im}
\DeclareMathOperator{\inv}{inv}
\DeclareMathOperator{\Tr}{Tr}
\DeclareMathOperator{\Spec}{Spec}

\newtheorem{theorem}{Theorem}[section]
\newtheorem{lemma}[theorem]{Lemma}

\title{Failure of strong approximation on an affine cone}
\author{Martin Bright}
\address{Mathematisch Instituut \\ Niels Bohrweg 1 \\ 2333 CA Leiden \\ Netherlands}
\email{m.j.bright@math.leidenuniv.nl}
\author{Ivo Kok}
\address{Mathematisch Instituut \\ Niels Bohrweg 1 \\ 2333 CA Leiden \\ Netherlands}
\email{ivokok@telfort.nl}


\begin{document}

\begin{abstract}
We use the Brauer--Manin obstruction to strong approximation on a punctured affine cone to explain
why some mod $p$ solutions to a homogeneous Diophantine equation of degree $2$ cannot be lifted to coprime integer solutions.
\end{abstract}

\maketitle

\section{Introduction}

Let $Y \subset \P^3_\Q$ be the quadric surface defined by the equation
\begin{equation}\label{eq:Y}
X_0^2 + 47 X_1^2 = 103 X_2^2 + (17 \times 47 \times 103) X_3^2.
\end{equation}
One can easily check that $Y$ is everywhere locally soluble, and so has rational points.
Being a quadric surface, $Y$ satisfies weak approximation.  In particular, if we fix a prime $p$, then any smooth point on the reduction of $Y$ at $p$ lifts to a rational point of $Y$.
Given that a point on the reduction of $Y$ is given by $(\tilde{x}_0, \tilde{x}_1, \tilde{x}_2, \tilde{x}_3) \in \F_p^4$ satisfying~\eqref{eq:Y}, and a point of $Y(\Q)$ can be given by coprime integers $(x_0,x_1,x_2,x_3) \in \Z^4$ satisfying~\eqref{eq:Y}, one might be tempted to think that every $\F_p$-solution $(\tilde{x}_0, \tilde{x}_1, \tilde{x}_2, \tilde{x}_3)$ can be lifted to a coprime integer solution $(x_0,x_1,x_2,x_3)$.

However, at the end of the article~\cite{Bright:vanishing}, it was remarked that $Y$ has the following interesting feature: if $(\tilde{x}_0, \tilde{x}_1, \tilde{x}_2, \tilde{x}_3)$ is a solution to~\eqref{eq:Y} over $\F_{17}$, then at most half of the non-zero scalar multiples of $(\tilde{x}_0,\tilde{x}_1,\tilde{x}_2,\tilde{x}_3) \in \F_{17}^4$ can be lifted to coprime 4-tuples $(x_0,x_1,x_2,x_3) \in \Z^4$ defining a point of $Y$.
That observation was a by-product of the calculation of the Brauer--Manin obstruction to rational points on a diagonal quartic surface related to $Y$.
In this note we will interpret the observation as a failure of strong approximation on the punctured affine cone over $Y$, and will show that this failure is itself due to a Brauer--Manin obstruction.

The same phenomenon has been observed by Lindqvist~\cite{Lindqvist} in the case of the quadric surface $X_0^2 - pqX_1^2 - X_2X_3$, for $p,q$ odd primes congruent to $1$ modulo $8$.  We expect that example also to be explained by a Brauer--Manin obstruction.

Following Colliot-Th\'el\`ene and Xu~\cite{CTX:AA-2013}, for a variety $X$ over $\Q$, we define $X(\ad_\Q)$ to be the set of adelic points of $X$, that is, the restricted product of $X(\Q_v)$ for all places $v$, with respect to the subsets $X(\Z_v)$.
(One needs to choose a model of $X$ to make sense of the notation $X(\Z_v)$, but since any two models agree outside a finite set of primes the resulting definition of $X(\ad_\Q^\infty)$ does not depend on the choice of model.)
Similarly, define $X(\ad_\Q^\infty)$ to be the set of adelic points of $X$ away from $\infty$, that is, the restricted product of $X(\Q_v)$ for $v \neq \infty$ with respect to the subsets $X(\Z_v)$.   Assuming that $X$ has points over every completion of $\Q$, we say that $X$ satisfies \emph{strong approximation away from $\infty$} if the image of the diagonal map $X(\Q) \to X(\ad_\Q^\infty)$ is dense.

If a variety $X$ does not satisfy strong approximation, this can sometimes be explained by a \emph{Brauer--Manin obstruction}.
Define
\[
X(\ad_\Q)^{\Br} = \{ (P_v) \in X(\ad_\Q) \mid \sum_v \inv_v A(P_v)=0 \text{ for all } A \in \Br X \},
\]
and define $X(\ad_\Q^\infty)^{\Br}$ to be the image of $X(\ad_\Q)^{\Br}$ under the natural projection map $X(\ad_\Q) \to X(\ad_\Q^\infty)$.  Then $X(\ad_\Q^\infty)^{\Br}$ is a closed subset of $X(\ad_\Q^\infty)$ that contains the image of $X(\Q)$.  If $X(\ad_\Q^\infty)^{\Br} \neq X(\ad_\Q^\infty)$, we say that there is a Brauer--Manin obstruction to strong approximation away from $\infty$ on $X$.

We now return to the variety $Y$ defined above.
Let $X \subset \A^4_\Q$ be the punctured affine cone over $Y$: that is, $X$ is the complement of the point $O = (0,0,0,0)$ in the affine variety defined by the equation~\eqref{eq:Y}.  There is a natural morphism $\pi \colon X \to Y$ given by restricting the usual quotient map $\A^4 \setminus \{O\} \to \P^3$, so that $X$ is realised as a $\mathbf{G}_{\mathrm{m}}$-torsor over $Y$.
To talk about integral points, we must choose a model: let $\X \subset \A^4_\Z$ be the complement of the section $(0,0,0,0)$ in the scheme defined by the equation~\eqref{eq:Y} over $\Z$.
If we let $f \in \Z[X_0,X_1,X_2,X_3]$ be the polynomial defining $Y$, then the integral points of $\X$ are given by
\[
\X(\Z) = \{ (x_0,x_1,x_2,x_3) \in \Z^4 \mid x_0,x_1,x_2,x_3 \text{ coprime, }f(x_0,x_1,x_2,x_3)=0 \}.
\]

\begin{theorem}
The group $\Br X / \Br \Q$ has order $2$; a generator is given by the quaternion algebra $(17,g)$, where $g \in \Z[X_0,X_1,X_2,X_3]$ is a homogeneous linear form defining the tangent hyperplane to $X$ at a rational point $P \in X(\Q)$.  There is a Brauer--Manin obstruction to strong approximation on $X$ away from $\infty$.  
Specifically, for any smooth point $\tilde{Q} \in \X(\F_{17})$, at most half of the scalar multiples of $\tilde{Q}$ lift to integer points of $\X$.
\end{theorem}

It is interesting to compare this result with the ``easy fibration method'' of~\cite[Proposition~3.1]{CTX:AA-2013}.  We have a fibration $\pi \colon X \to Y$, and the base $Y$ satisfies strong approximation.  However, the fibres are isomorphic to $\mathbf{G}_{\mathrm{m}}$, which drastically fails to satisfy strong approximation, so we cannot use that method to conclude anything about strong approximation on $X$.

\section{Quadric surfaces}

In this section we gather some basic facts about quadric surfaces.  Any non-singular quadric surface $Y \subset \P^3$ over a field $k$ of characteristic different from $2$ may be defined by an equation of the form $\mathbf{x}^T \mathbf{M} \mathbf{x} = 0$, where $\mathbf{M}$ is an invertible $4\times 4$ matrix with entries in $k$.  We define $\Delta_Y \in k^\times / (k^\times)^2$ to be the class of the determinant of $\mathbf{M}$, which is easily seen to be invariant under linear changes of coordinates.
If $\kbar$ is an algebraic closure of $k$ and $\bar{Y}$ is the base change of $Y$ to $\kbar$, then $\Pic\bar{Y}$ is isomorphic to $\Z^2$, generated by the classes of the two families of lines on $\bar{Y}$ \cite[Example~II.6.6.1]{Hartshorne:AG}.

\begin{lemma}\label{lem:quad}
Let $k$ be a field of characteristic not equal to $2$, and let $Y \subset \P^3_k$ be a non-singular quadric surface.  Then the two families of lines on $\bar{Y}$ are defined over the field $k(\sqrt{\Delta_Y})$, and are conjugate to each other.
\end{lemma}
\begin{proof}
We may assume that the matrix $\mathbf{M}$ defining $Y$ is diagonal, with entries $p,q,r,s$.  Following~\cite[Section~IV.3.2]{EH:GS}, we explicitly compute an open subvariety of the Fano scheme of lines on $Y$ by calculating the conditions for the line through $(1:0:a:b)$ and $(0:1:c:d)$ to lie in $Y$.  The resulting affine piece of the Fano scheme is given by
\[
\{ p + ra^2 + sb^2 = 0, rac+ sbd = 0, q + rc^2 + sd^2 = 0 \} \subset \A^4_k = \Spec k[a,b,c,d].
\]
This is easily verified to consist of two geometric components, each a plane conic, one contained in the plane $qra=-\sqrt{\Delta_Y} d, qsb = \sqrt{\Delta_Y} c$ and the other in the conjugate plane.
\end{proof}

\begin{lemma}\label{lem:weil}
Let $Y$ be a non-singular quadric surface over the finite field $\F_q$, with $q$ odd.  Then
\[
\# Y(\F_q) = \begin{cases}
q^2+2q+1 & \text{if $\Delta_Y \in (\F_q^\times)^2$;} \\
q^2+1 & \text{otherwise.}
\end{cases}
\]
\end{lemma}
\begin{proof}
This can be computed directly, but we recall how to obtain it from the Lefschetz trace formula for \'etale cohomology.
Let $\ell$ be a prime not equal to $p$.  Let $\bar{\F}_q$ be an algebraic closure of $\F_q$, let $\bar{Y}$ be the base change of $Y$ to $\bar{\F}_q$, and let $F \colon \bar{Y} \to \bar{Y}$ be the Frobenius morphism.  The Lefschetz trace formula states that $\# Y(\F_q)$ can be calculated as
\[
\# Y(\F_q) = \sum_{i = 0}^4 (-1)^i \Tr(F^* | \H^i(\bar{Y}, \Q_\ell)).
\]
Because $Y$ is smooth and projective, there are isomorphisms of Galois modules $\H^0(\bar{Y},\Q_\ell) \cong \Q_\ell$ and $\H^4(\bar{Y},\Q_\ell) \cong \Q_\ell(-2)$ (see~\cite[VI.11.1]{Milne:EC}).
We have $\bar{Y} \cong \P^1 \times \P^1$.  The standard calculation of the cohomology groups of projective space~\cite[VI.5.6]{Milne:EC},
and the K\"unneth formula~\cite[Corollary~VI.8.13]{Milne:EC}, give
$\H^i(\bar{Y},\Q_\ell)=0$ for $i$ odd, and show that $\H^2(\bar{Y},\Q_\ell)$ has dimension $2$.
This reduces the formula to
\[
\# Y(\F_q) = q^2 + 1 + \Tr(F^* | \H^2(\bar{Y},\Q_\ell)).
\]
Moreover, the cycle class map (arising from the Kummer sequence) gives a Galois-equivariant injective homomorphism
\[
\Pic \bar{Y} \otimes_\Z \Q_\ell \to \H^2(\bar{Y},\Q_\ell(1)),
\]
which by counting dimensions must be an isomorphism.
If $\Delta_Y$ is a square in $\F_q$, then the Galois action is trivial and we obtain (after twisting) $\Tr(F^* | \H^2(\bar{Y},\Q_\ell)) = 2q$.  If $\Delta_Y$ is not a square in $\F_q$, then $F^*$ acts on $\Pic \bar{Y} \cong \Z^2$ by switching the two factors, so with trace zero.  In either case we obtain the claimed number of points.  (Note that, in the first case, $Y$ is isomorphic to $\P^1 \times \P^1$, so we should not be surprised that it has $(q+1)^2$ points.) 
\end{proof}

\section{Proof of the theorem}

Firstly, we calculate the Brauer group of $X$; it is convenient to do so in more generality.

\begin{lemma}\label{lem:br}
Let $k$ be a field of characteristic zero, let $Y \subset \P^3_k$ be a smooth quadric surface, and let $X \subset \A^4_k$ be the punctured affine cone over $Y$.
If $\Delta_Y \in (k^\times)^2$, then we have $\Br X = \Br k$. 
Otherwise, suppose that $X$ has a $k$-rational point $P$, and let $g$ be a homogeneous linear form defining the tangent hyperplane to $X$ at $P$.  Then $\Br X / \Br k$ has order $2$, and is generated by the class of the quaternion algebra $(\Delta_Y, g)$.  This class does not depend on the choice of $P$.
\end{lemma}
\begin{proof}
Let $\kbar$ be an algebraic closure of $k$, and
let $\bar{X}$ and $\bar{Y}$ denote the base changes to $\kbar$ of $X$ and $Y$, respectively.
By~\cite[Theorem~2.2]{Ford:JPAM-2001}, we have $\Br(\bar{X}) \cong \Br(\bar{Y})$; but $\bar{Y}$ is a rational variety, so its Brauer group is trivial.   So it remains to compute the algebraic Brauer group of $X$.

We claim that there are no non-constant invertible regular functions on $X$.  
Indeed, let $C \subset \A^4_k$ be the (non-punctured) affine cone over $Y$.  Because $C$ is Cohen--Macaulay and $(0,0,0,0)$ is of codimension $\ge 2$ in $C$, we have
\[
k[X] = k[C] = k[X_0,X_1,X_2,X_3] / (f)
\]
where $f$ is the homogeneous polynomial defining $Y$.  This is a graded ring and its invertible elements must all have degree $0$, so are constant.

The Hochschild--Serre spectral sequence gives an injection $\Br X / \Br k \to \H^1(k, \Pic \bar{X})$.  (Here we use $k[X]^\times = k^\times$ and $\Br \bar{X}=0$.)  By~\cite[Exercise~II.6.3]{Hartshorne:AG}, there is an exact sequence
\[
0 \to \Z \to \Pic \bar{Y} \xrightarrow{\pi^*} \Pic \bar{X} \to 0,
\]
where $\pi \colon X \to Y$ is the natural projection and the first map sends $1$ to the class of a hyperplane section of $\bar{Y}$.  Using Lemma~\ref{lem:quad} shows that $\Pic \bar{X}$ is isomorphic to $\Z$, with $G=\Gal(k(\sqrt{\Delta_Y})/k)$ acting by $-1$.
The inflation-restriction sequence shows $\H^1(k,\Pic \bar{X}) \cong \H^1(G,\Pic \bar{X})$.  If $\Delta_Y$ is a square, then this group is trivial, and we conclude that $\Br X / \Br k$ is also trivial.  
Otherwise $G=\{1,\sigma\}$ has order $2$, and we have
\[
\H^1(G,\Pic\bar{X}) \cong \hat{\H}^{-1}(G,\Pic\bar{X}) = \frac{\ker(1+\sigma)}{\im(1-\sigma)} = \frac{\Z}{2\Z}.
\]

To conclude, it is sufficient to show that the algebra $(\Delta_Y,g)$ is non-trivial in $\Br X / \Br k$.  Because the polynomial $g$ also defines the tangent plane to $Y$ at $\pi(P)$, the divisor $(g)$ is equal to $\pi^*(L+L')$, where $L$ is a line passing through $\pi(P)$ and $L'$ is its conjugate.  By~\cite[Proposition~4.17]{Bright:thesis}, this shows that $(\Delta_Y,g)$ is a non-trivial element of order $2$ in $\Br X / \Br k$.  (The reference works with a smooth projective variety, but the proof generalises easily to any smooth $X$ with $k[X]^\times = k^\times$.) 
\end{proof}

We now return to the specific case where $X$ is the punctured affine cone over the quadric surface defined by the equation~\eqref{eq:Y}.  We will need to be more careful about constant algebras than we have been up to this point.  
Recall that $\X(\Z)$ consists of points $P=(x_0,x_1,x_2,x_3)$ where $x_0,x_1,x_2,x_3$ are coprime integers satisfying the equation~\eqref{eq:Y}.  Given such a $P$, we define the linear form
\[
\ell_P = x_0 X_0 + 47 x_1 X_1 - 103 x_2 X_2 - (17 \times 47 \times 103) x_3 X_3 \in \Z[X_0,X_1,X_2,X_3]
\]
and the quaternion algebra $A_P = (17,\ell_P) \in \Br X$.
Note that the linear form $\ell_P$ does indeed define the tangent plane to $X$ at $P$, so Lemma~\ref{lem:br} shows that $A_P$ represents the unique non-trivial class in $\Br X / \Br \Q$.  We will now evaluate the Brauer--Manin obstruction associated to $A_P$.

\begin{lemma}\label{lem:easy}
Fix $P \in \X(\Z)$.  Then, for any place $v$ of $\Q$ for which $17$ is a square in $\Q_v$, we have $\inv_v A_P(Q)=0$ for all $Q \in X(\Q_v)$.
\end{lemma}
\begin{proof}
The homomorphism $\Br X \to \Br \Q_v$ given by evaluation at $Q$ factors through $\Br(X \times_\Q \Q_v)$, but the image of $A_P$ in this group is zero.
\end{proof}

Note that Lemma~\ref{lem:easy} applies in particular to $v=\infty$, $v=2$, $v=47$ and $v=103$.

For the following lemma, let $\Y$ be the model for $Y$ over $\Z$ defined by the equation~\eqref{eq:Y}, and extend $\pi$ to the natural projection $\X \to \Y$.

\begin{lemma}\label{lem:good}
Fix $P \in \X(\Z)$.  Let $p \neq 17$ be a prime such that $17$ is not a square in $\Q_p$, and let $Q \in \X(\Z_p)$ be such that $\pi(Q) \not\equiv \pi(P) \pmod{p}$.   Then $\inv_p A_P(Q)=0$.
\end{lemma}
\begin{proof}
If $\ell_P(Q)$ is not divisible by $p$, then $\ell_P(Q)$ is a norm from the unramified extension $\Q_p(\sqrt{17})/\Q_p$ and therefore we have $\inv_p A_P(Q)=0$.  

Now suppose that $\ell_P(Q)$ is divisible by $p$.  Denote by $\tilde{Y}$ the base change of $\Y$ to $\F_p$.  Let $\tilde{P},\tilde{Q} \in \tilde{Y}(\F_p)$ be the reductions modulo $p$ of $\pi(P),\pi(Q)$ respectively.
The variety $\tilde{Y}$ is a smooth quadric over $\F_p$, and the tangent space $T_{\tilde{P}} \tilde{Y}$ is cut out by the reduction modulo $p$ of the linear form $\ell_P$.  
By Lemma~\ref{lem:quad}, the scheme $\tilde{Y} \cap \{ \ell_P = 0 \}$ consists of two lines that are conjugate over $\F_p(\sqrt{17})$.  Therefore the only point of $\tilde{Y}(\F_p)$ at which $\ell_P$ vanishes is $\tilde{P}$.  It follows that $\ell_P(Q)$ can only be divisible by $p$ if $\tilde{Q}$ coincides with $\tilde{P}$.
\end{proof}

\begin{lemma}\label{lem:same}
Let $P,P' \in \X(\Z)$ be two points.  Then $A_P$ and $A_{P'}$ lie in the same class in $\Br X$.
\end{lemma}
\begin{proof}
By Lemma~\ref{lem:br}, we already know that the difference $A = A_P - A_{P'}$ lies in $\Br \Q$.  It will be enough to show that $\inv_v A = 0$ for $v \neq 17$, for then the product formula shows $\inv_{17} A = 0$ also, and therefore $A=0$.

For $v$ for which $17$ is a square in $\Q_v$, take $Q$ to be any point of $X(\Q_v)$; then Lemma~\ref{lem:easy} shows $\inv_v A_P(Q) = \inv_v A_{P'}(Q) = 0$ and therefore $\inv_v A = 0$.

For $p \neq 17$ such that $17$ is not a square in $\Q_p$, Lemma~\ref{lem:weil} shows that $\tilde{Y} = \Y \times_\Z \F_p$ contains a point $\tilde{Q}$ that is equal neither to $\pi(P)$ nor to $\pi(P')$ modulo $p$.  Hensel's Lemma shows that $\tilde{Q}$ lifts to a point $Q \in \X(\Z_p)$.  Lemma~\ref{lem:good} shows $\inv_p A_P(Q) = \inv_p A_{P'}(Q) = 0$, so again we have $\inv_p A = 0$, completing the proof.
\end{proof}

\begin{lemma}
Fix $P \in \X(\Z)$.  For $p \neq 17$, we have $\inv_v A_P(Q)=0$ for all $Q \in \X(\Z_p)$.
\end{lemma}
\begin{proof}
If $17$ is a square in $\Q_p$, then this follows from Lemma~\ref{lem:easy}.  Otherwise, Lemma~\ref{lem:weil} shows that $\tilde{Y} = \Y \times_\Z \F_p$ contains at least two points.  Weak approximation on $Y$ then gives a point $P' \in \X(\Z)$ such that $\pi(P)$ and $\pi(P')$ are different modulo $p$.  By Lemma~\ref{lem:same}, the algebras $A_P$ and $A_{P'}$ lie in the same class in $\Br X$, so it does not matter which we use to evaluate the invariant.  Lemma~\ref{lem:good} then gives the result.
\end{proof}

It remains to evaluate the invariant at $17$.  In the following lemma, if $Q = (y_0,y_1,y_2,y_3)$ is a point of $\X$, then $\lambda Q$ denotes the point $(\lambda y_0, \lambda y_1, \lambda y_2, \lambda y_3)$.

\begin{lemma}
Fix $P \in \X(\Z)$ and $Q \in \X(\Z_{17})$.  For any $\lambda \in \Z_{17}^\times$, having reduction $\tilde{\lambda} \in \F_{17}^\times$, we have 
\[
\inv_{17} A_P(\lambda Q) = \begin{cases}
\inv_{17} A_P(Q) & \text{if $\tilde{\lambda} \in (\F_{17}^\times)^2$} \\ 
\inv_{17} A_P(Q) + \frac{1}{2} & \text{otherwise.}
\end{cases}
\]
\end{lemma}
\begin{proof}
Suppose that $\ell_P(Q)$ is non-zero.
Because $\ell_P$ is homogeneous of degree $1$, we have
\[
\inv_{17} A_P(\lambda Q) = \inv_{17} (17, \lambda \ell_P(Q)) = \inv_{17} A_P(Q) + \inv_{17} (17, \lambda).
\]
But $\inv_{17}(17, \lambda)$ is zero if and only if $\tilde{\lambda}$ is a square in $\F_{17}^\times$.

If $\ell_P(Q)$ is zero, then Lemma~\ref{lem:quad} shows that we have $\pi(P) = \pi(Q)$.  Using weak approximation on $Y$, we can find a point $P' \in \X(\Z)$ with $\pi(P') \neq \pi(Q)$, and Lemma~\ref{lem:same} shows that replacing $A_P$ by $A_P'$ gives the same invariant.
\end{proof}

Note that the only singular points of $\X \times_\Z \F_{17}$ are those of the form $(0,0,0,a)$, and these do not lift to points of $\X(\Z_{17})$.  So the smooth points of $\X(\F_{17})$ are precisely those that lift to $\X(\Z_{17})$.

Putting all these calculations together proves the following.  Let $U \subset X(\ad_\Q^\infty)$ be the open subset defined as
\[
U = \prod_{p \neq 17} \X(\Z_p) \times \{ Q \in \X(\Z_{17}) \mid \inv_{17} A_P(Q) = \frac{1}{2} \}.
\]
Then $U$ is a non-empty open subset that does not meet $X(\ad_\Q^\infty)^{\Br}$, showing that there is a Brauer--Manin obstruction to strong approximation away from $\infty$ on $X$.  More specifically, for any smooth point $\tilde{Q} \in \X(\F_{17})$, half of the scalar multiples of $\tilde{Q}$ lie in the image of $U$, showing that they do not lift to integer points of $\X$.

\bibliographystyle{abbrv}
\bibliography{martin}

\end{document}